\documentclass{aptpub}
\usepackage[T1]{fontenc}
\usepackage[latin1]{inputenc}
\usepackage{ucs}
\usepackage[english]{babel}
\usepackage{multicol}
\usepackage{amssymb}
\usepackage{amsmath}
\usepackage{float}
\usepackage{theorem}
\usepackage{graphicx, epsfig}
\usepackage{verbatim}
\usepackage{bbm}

\usepackage[official,right]{eurosym}

\newtheorem{defi}{Definition}
\newcommand{\id}[1]{\ensuremath{\mathbbm{1}_{#1}}}
\newcommand{\bc}{\begin{center}}
\newcommand{\ec}{\end{center}}
\newcommand{\1}{\mathbf{1}}

\newcommand{\Esp}{\mathbb{E}}

\renewcommand{\P}{\mathbb{P}}

\newcommand{\R}{\mathbb{R}}

\newcommand{\Nat}{\mathbb{N}}

\newcommand{\cC}{\mathcal{C}}

\newcommand{\cN}{\mathcal{N}}

\newcommand{\Z}{\mathbb{Z}}
\newcommand{\hZ}{\widehat{\mathbb{Z}}}

\newcommand{\Wsdn}{\mathbb{W}^{SD,n}}

\newcommand{\W}{\mathbb{W}}

\newcommand{\pare}[1]{\left ( #1 \right )}

\newcommand{\sgn}{{\rm sgn}}

\def \put{{\rm Put}}

\def \sign{{\rm Sgn}}

\def \defi{:=}

\authornames{P. \'Etoré and M. Martinez} % insert the authors here for use in running head
\shorttitle{Exact simulation for SDE with discontinuous drift} 

\begin{document}

\title{Exact simulation for solutions of one-dimensional Stochastic Differential Equations with discontinuous drift}

\authorone[ENSIMAG - Laboratoire Jean Kuntzmann ]{Pierre \'Etoré}
\addressone{Tour IRMA 51, rue des Math\'ematiques, 38041 Grenoble Cedex 9, France. email: pierre.etore@imag.fr, Phone: + 33 (0)4 76 51 45 57 }

\authortwo[Universit\'e Paris-Est Marne-La-Vallée, Laboratoire d'Analyse et de Math\'ematiques Appliqu\'ees, UMR $8050$]{Miguel Martinez}
\addresstwo{
5  Bld Descartes, Champs-sur-Marne, 77454 Marne-la-Vall\'ee Cedex 2, France. email: miguel.martinez@univ-mlv.fr 
}

\begin{abstract}
In this note we propose an exact simulation algorithm for the solution of
\begin{equation}
\label{eds-intro}
dX_t=dW_t+\bar{b}(X_t)dt,\quad X_0=x,
\end{equation}
where $\bar{b}$ is a smooth real function except at point $0$ where $\bar{b}(0+)\neq \bar{b}(0-)$. The main idea is to sample an exact skeleton of $X$ using an algorithm deduced from the convergence of the solutions of the skew perturbed equation
 \begin{equation}
\label{edsbeta}
dX^\beta_t=dW_t+\bar{b}(X^\beta_t)dt + \beta dL^0_t(X^\beta),\quad X_0=x
\end{equation}
towards $X$ solution of \eqref{eds-intro} as $\beta \neq 0$ tends to $0$. 

In this note, we show that this convergence induces the convergence of exact simulation algorithms proposed by the authors in \cite{etoremartinez1} for the solutions of \eqref{edsbeta} towards a limit algorithm. Thanks to stability properties of the rejection procedures involved as $\beta$ tends to $0$, we prove that this limit algorithm is an exact simulation algorithm for the solution of the limit equation \eqref{eds-intro}. Numerical examples are shown to illustrate the performance of this exact simulation algorithm.

%% Text of abstract

\end{abstract}

\keywords{
Exact simulation methods ;  Brownian motion with Two-Valued Drift ; One-dimensional diffusion ; Skew Brownian motion ; Local Time.
}

\ams{65C05,65U20}{65C30,65C20}

%%%%%%%%%%%%%%%%%%%%%%%%%%%%%%%%%%%%%%%%%%%%%%%%%%%%%%%%%%%%%%%
%%%%%%%%%%%%%%%%%%%%%%%%%%%%%%%%%%%%%%%%%%%%%%%%%%%%%%%%%%%%%%%
%%%%%%%%%%%%%%%%%%%% --------  INTRODUCTION ----- %%%%%%%%%%%%%%%%%%%%%%%%%%%%%
%%%%%%%%%%%%%%%%%%%%%%%%%%%%%%%%%%%%%%%%%%%%%%%%%%%%%%%%%%%%%%%
%%%%%%%%%%%%%%%%%%%%%%%%%%%%%%%%%%%%%%%%%%%%%%%%%%%%%%%%%%%%%%%
\section{Introduction}
\label{sec:intro}
\subsection{Motivations and exposition of the problem}
Exact simulation methods for trajectories of one-dimensional SDEs has been a subject of much interest in the last years : see for example \cite{beskos1}, \cite{beskos2}, \cite{beskos5}, \cite{tanre1}, \cite{sbai}. Unlike the classical simulation methods,
which all involve some kind of discretization error (see for example \cite{ball:tala:96:1} for the Euler Scheme), the exact simulation methods are constructed in such a way that they do not present any discretization error (under the strong hypothesis that the diffusion coefficient is constant and equal to one). In the last years, the original method presented in the
fundamental article \cite{beskos2} has been extended to overcome various limitations of the initial algorithm ; it has been 
generalized to include the cases of unbounded drifts (\cite{beskos5}, \cite{beskos4}) and extended to various 'non classical' type of SDE (\cite{etoremartinez1}).

In this paper, we present an attempt for the adaptation of the exact simulation methods of \cite{beskos2} to one-dimensional SDEs that possess a discontinuous drift at point $0$. Namely, our object of study is $(X_t)_{t\geq 0}$ solution of 
\begin{equation}
\label{SDE-intro}
dX_t=W_t+\bar{b}(X_t)dt,\hspace{1,0 cm}X_0=x,
\end{equation}
where $\bar{b}$ is a smooth real function except possibly at point $0$ where $\bar{b}(0+)\neq \bar{b}(0-)$.

The simplest case of a process solution of an equation of type \eqref{SDE-intro} is surely the so-called 'Brownian motion with two valued drift' solution of
\begin{equation}
\label{Bang-Bang-intro}
dX_t=W_t+\pare{\theta_0\id{X_t>0} +\theta_1\id{X_t<0}} dt,\hspace{1,0 cm}X_0=x,
\end{equation}
where $\pare{\theta_0,\theta_1}\in \R^2$.
For a general reference concerning these types of motions, we refer to \cite{kara} p.440-441 or \cite{kara-2}. These motions appear in stochastic control problems (see for example \cite{Benes}, \cite{kara-2}) and also theoretical studies concerning representations of reflected Brownian motion with drift (see \cite{Shiryaev} in the case $\theta_0=-\theta_1$). Even though there exist explicit representation formulae for the densities of such Brownian motions with two valued drift in terms of combination of convolution integrals (see \cite{kara} p.440-441), up to our knowledge there is no exact numerical simulation algorithm for such motions available in the literature. The algorithm presented in this paper gives an answer to this question.

\subsection{Main ideas of the paper}
In \cite{etoremartinez1}, the authors manage to adapt the exact simulation methods of \cite{beskos2} to the case of one-dimensional SDEs that possess an additional term involving the local time of the unknown process at point $0$. Namely, the exact simulation methods of \cite{beskos2} are modified in \cite{etoremartinez1} to include the case where $(X^\beta_t)_{t\geq 0}$ is the solution of 
\begin{equation}
\label{SDE-beta-intro}
dX^\beta_t =W_t+\bar{b}(X^\beta_t)dt+\beta dL^0_t(X^\beta),\hspace{1,0 cm}X_0=x.
\end{equation}
In this situation $0\neq |\beta|<1$, $L^0_t(X)$ denotes the symmetric local time of $X^ \beta$ in zero at time $t$, and $\bar{b}$ is still allowed to be discontinuous at $0$.

The main idea in \cite{etoremartinez1} was to propose an exact rejection simulation algorithm for the solutions of \eqref{SDE-beta-intro} using as sampling reference measure the law of some drifted skew Brownian motion with prescribed terminal distribution and with drift of magnitude $1/\beta$ avoiding the case where $\beta=0$, for which we propose a proper treatment here. Our contribution in \cite{etoremartinez1} deals mainly on the simulation of  bridges of such drifted skew Brownian motions using a classical rejection procedure and looking for tractable rejection functions. 

Unfortunately, a direct exact simulation method along the same lines as \cite{etoremartinez1} cannot be properly defined in the case where $\beta=0$. However, we know from Le Gall in \cite{legall} that $X^\beta$ solution of \eqref{SDE-beta-intro} tends strongly to $X$ the solution of \eqref{SDE-intro} as $\beta$ tends to $0$ on each time interval $[0,T]$. This leads us to examine what happens at the level of the algorithms proposed in \cite{etoremartinez1} as $\beta$ tends to $0$.

%%%%%Indeed, this is what we check by direct computation. Let us roughly explain where the computational miracle comes from. 
%%When we want to simulate exactly the solution of \eqref{SDE-beta-intro} along the lines of \cite{etoremartinez1}, the reference measure becomes then that of a bridge of a skew Brownian motion with a constant drift proportional to $1/\beta$. So that as $\beta$ tends to $0$, there is no strong convergence of such bridges because the drift $1/\beta$ explodes. Nevertheless, would we study the case of a plain standard Brownian motion with constant drift $\mu$, we would surely point out that even though there is no strong convergence of such 
%%processes as $\mu$ tends to $\infty$, the bridges of the standard Brownian motion with constant drift $\mu$ are stable in law as $\mu$ tends to $\infty$. In fact, it is known and may be briefly checked that the laws remain constant and are exactly the same as that of a standard Brownian bridge with no drift at all ! Analogously, this stability phenomenon is similar to what we check directly by computations when we look at the finite dimensional laws of bridges of a drifted skewed Brownian motion with parameter $\beta$ and drift proportional to $1/\beta$ as $\beta$ tends to $0$.

In fact, we check here by computations that there is indeed a convergence phenomenon at the level of rejection functions and rejection sets involved in the exact simulation algorithms given in \cite{etoremartinez1}. This convergence gives rise naturally to a nice and implementable limiting algorithm. 

The main problem becomes then to prove rigorously that this limiting algorithm is indeed an exact simulation algorithm for the solution of \eqref{SDE-intro}. In particular, as far as we see, the direct interpretation of this limiting algorithm is not clear ; for the time being, we have to confess that we really understand the construction of the limiting algorithm exposed in this paper only {\it via} the convergence procedure explained above. Let us also emphasize that this new algorithm is still a rejection algorithm, and one may naturally ask for a direct interpretation of its corresponding reference measure. In Remark \ref{remarque-importante} we give an interpretation of the reference measure (corresponding to the limit rejection algorithm) in terms of a standard Brownian motion conditioned on prescribed laws for its final position and its local time at $0$ at time horizon $T$.

\subsection{Outline of the paper} 
The paper is organized as follows. In the preliminary Section 2, we explain the convergence of rejection sampling algorithms in a general framework. The result exposed in this section will be used to justify that our limiting algorithm is indeed an exact simulation algorithm for the solution of \eqref{SDE-intro}. The exact simulation problem treated here is presented in Section 3, where we explain the manner in which we adapt the exact simulation methods of \cite{beskos2} to our situation. Yet, the resulting algorithm adapted from \cite{beskos2} is not directly implementable in our context because we have to sample from a complicated reference probability measure $\hZ$. The sections 4 and 5 are devoted to the interpretation of $\hZ$ as a limit of some sequence $\pare{\hZ_n}$ of better known probability measures. Finally in Section 6, we apply the results of the preliminary Section 2 to the sequence $\pare{\hZ_n}$. This gives rise to a directly implementable limit algorithm 
 for the exact simulation of a skeleton along the reference probability $\hZ$. We end up the article with numerical results and illustrative examples shown in Section 7.
 
\section{Preliminary~:~convergence of abstract rejection sampling algorithms}

%Following the spirit of \cite{beskos1,beskos2}, we achieve here exact sampling for our SDE of interest using rejection sampling tools. We sum up in the following proposition the results on rejection sampling we will need.

\begin{prop}
\label{prop-delamort}
i) Assume that we have a sequence $(\xi_n)$ of probability measures on a measurable space $(S,\mathcal{S})$, and $\xi_{dom}$ a probability measure on $(S,\mathcal{S})$, satisfying for any $n\in\Nat$
$$
\frac{d\xi_n}{d\xi_{dom}}=\frac{1}{\varepsilon_n}f_n,$$
with $\varepsilon_n>0$ and $0\leq f_n\leq 1$.

Assume that $f_n\to f$ as $n\to\infty$ point-wise on $S$. 

Then, $\pare{\xi_n}$ converges towards a probability measure $\xi$ satisfying
\begin{equation}
\label{dens-propmort}
\frac{d\xi}{d\xi_{dom}}=\frac{1}{\varepsilon}f,
\end{equation}
with $\varepsilon=\lim_{n\to\infty}\varepsilon_n$.

\vspace{0.3cm}
ii) Moreover, let $(Y_k,I_k)_{k\geq 1}$ be a sequence of i.i.d. random elements taking values in $S\times\{0,1\}$
such that $Y_1\sim\xi_{dom}$ and $\P[I_1=1|Y_1=y]=f(y)$ for all $y\in S$. Define $\tau\defi\min(k\geq 1=I_k=1)$. Then, $\P(Y_\tau\in dy)=\xi(dy)$.

\end{prop}
\begin{proof}
For any $A\in\mathcal{S}$ we have $\xi_n(A)=\frac{1}{\varepsilon_n}\int_Af_n(z)\xi_{dom}(dz)$. By dominated convergence we have
 $$
 \int_Af_n(z)\xi_{dom}(dz)\xrightarrow[n\to\infty]{}\int_Af(z)\xi_{dom}(dz).$$
 Taking $A=S$, and as $\xi_n(S)=1$ for any $n\in\Nat$, we have
 $$
 \varepsilon_n=\dfrac{1}{\int_Sf_n(z)\xi_{dom}(dz)}\xrightarrow[n\to\infty]{}\dfrac{1}{\int_Sf(z)\xi_{dom}(dz)}=:\varepsilon.$$
 
 Setting now for any $A\in\mathcal{S}$, $\xi(A)\defi\frac {1} {\varepsilon} \int_Af(z)\xi_{dom}(dz)$, it is clear that
 $$
 \forall A\in\mathcal{S},\quad\xi_n(A)\xrightarrow[n\to\infty]{}\xi(A).$$
 Then $\xi$ is a probability measure on $(S,\mathcal{S})$. It satisfies \eqref{dens-propmort} by construction. This proves point i). For the proof of point ii), see Proposition 1 in \cite{beskos1}.
\end{proof}

\section{Exact sampling algorithm for a SDE with discontinuous drift (inspired by \cite{beskos2})}
\subsection{Assumptions}
\label{sub-sec:assumptions}
The function $\bar{b}:\R\to\R$ is bounded with bounded first derivative on ${\mathbb R}^{\ast,+}$ and ${\mathbb R}^{\ast,-}$ with a possible discontinuity at point $\{0\}$.
We set $M$ a constant such that
\begin{equation}
\label{eq:M}
\sup_{z\in\R}|\bar{b}(z)|\leq M.
\end{equation}
We suppose that both limits $\lim_{z\rightarrow 0+}\bar{b}(z)=\bar{b}(0+)$ and $\lim_{z\rightarrow 0-}\bar{b}(z)=\bar{b}(0-)$ exist and are finite. The value $\bar{b}(0)$ of the function $\bar{b}$ at $0$ is
of no importance and can be fixed arbitrarily to some constant (possibly different from either $\bar{b}(0+)$ or $\bar{b}(0-)$).

We introduce the notation
\begin{equation}
\label{eq-theta}
\theta\defi \frac{\bar{b}(0+)-\bar{b}(0-)}{2}.
\end{equation}

\subsection{Change of probability}
Let $0<T<\infty$. Denote $C=C([0,T],\R)$ the set of continuous mappings from $[0,T]$ to $\R$ and $\cC$ the Borel 
$\sigma$-field on $C$ induced by the supreme norm.

Let $\P$ be a probability measure on $(C,\cC)$ and $W$ a Brownian motion under $\P$ together with its completed natural filtration $\pare{{\cal F}_t}_{t\geq 0}$. We will denote $\P^{x}=\P\pare{\cdot\,|\,W_0=x}$. When necessary we will denote by $\omega=(\omega_t)_{0\leq t\leq T}$ the coordinate process.

We consider the following SDE 
\begin{equation}
\label{eds}
dX_t=dW_t+\bar{b}(X_t)dt,\quad X_0=x.
\end{equation}
Our objective is to sample along $X_T$.

Let us define on $(C,\cC)$ the probability measure $\W$ by
$$
\frac{d\W}{d\P}=\exp\Big\{-\int_0^T\bar{b}(X_t)dW_t - \frac 1 2 \int_0^T\bar{b}^2(X_t)dt  \Big\}.$$
(Note that the assumptions in § \ref{sub-sec:assumptions} ensure that $\W$ is well defined).

Under $\W$ the process $X$ is a Brownian motion and we have,
$$
\frac{d\P}{d\W}=\exp\Big\{\int_0^T\bar{b}(X_t)dX_t - \frac 1 2 \int_0^T\bar{b}^2(X_t)dt  \Big\}.$$

Thus for any bounded  continuous functional $F:(C,\cC) \to \R$ we have,
\begin{equation}
\label{eq-ef}
\Esp_{\P}^x[F(X)]=\Esp_{\W}^x\big[F(X)\exp\big\{ \int_0^T\bar{b}(X_t)dX_t - \frac 1 2 \int_0^T\bar{b}^2(X_t)dt   \big\}  \big].
\end{equation}

We set $B(x)=\int_0^x\bar{b}(y)dy$. Using the symmetric Itô-Tanaka  formula (see Exercise VI-1-25 in \cite{RY}), and the Occupation times formula (\cite{RY}) we get
$$
B(X_T)-B(X_0)=\int_0^T\bar{b}(X_t)dX_t+\frac{1}{2} \int_0^T\bar{b}'(X_t)dt+\frac{\bar{b}(0+)-\bar{b}(0-)}{2}L^0_T(X),
$$
Thus \eqref{eq-ef} becomes,
$$
\Esp_{\P}^x[F(X)]=\Esp_{\W}^x\big[F(X)\exp\big\{ B(X_T)-B(x) - \theta L^0_T(X) -  \int_0^T\phi(X_t)dt   \big\}  \big],$$
where we have set
$$
\phi(x)\defi \frac{\bar{b}^2(x)+\bar{b}'(x)}{2}.$$
Setting now
$$
\tilde{\phi}(x)=\phi(x)-\inf_{x\in\R}\phi(x),$$
we finally get that for any bounded and continuous functional $F:(C,\cC) \to\R$ we have,
$$
\Esp_{\P}^x[F(X)]\propto\Esp_{\W}^x\big[F(X)\exp\big\{ B(X_T)-B(x) - \theta L^0_T(X)\big\} \exp\big\{-  \int_0^T\tilde{\phi}(X_t)dt   \big\}  \big].
$$

Let us now introduce the probability measure $\Z$ on $(C,\cC)$ defined in the following way
$$
\frac{d\Z}{d\W}(\omega)\propto\exp\big\{ B(X_T(\omega))-B(x) - \theta L^0_T(X)(\omega)\big\}.$$
Under the assumptions of § \ref{sub-sec:assumptions}, $\Z$ is well defined. 
 
 In the sequel we note $\hZ$ the probability measure induced on $(C,\cC)$ by the law of $X$ under $\Z$.
 We have
 \begin{equation}
 \label{eq-dens-loiX-hZ}
 \Esp_{\P}^x[F(X)]=c\,\Esp^x_{\hZ}\big[F(\omega) \exp\big\{-  \int_0^T\tilde{\phi}(\omega_t)dt   \big\}  \big],
 \end{equation}
 where $c$ is a normalizing constant (we make it explicit in the expression above for the purpose of proving Proposition \ref{prop-convZZn} below).
 
\begin{rem} (Interpretation of the probability $\hZ$)
\label{remarque-importante}

Recall that under $\W$ the process $X$ is a Brownian motion and that, by definition,
$$
\frac{d\Z}{d\W}(\omega)\propto\exp\big\{ B(X_T(\omega))-B(x) - \theta L^0_T(X)(\omega)\big\}.$$
In particular, under the probability $\Z$, $X$ is a Brownian motion conditioned on $(X_T,L^0_T)\sim h(y, \ell)dyd\ell$ with
\begin{equation*}
h(y,\ell)dyd\ell \propto\exp\pare{B(y)-B(x) - \theta \ell}{\mathbb W}^{x}\pare{X_T\in dy, L^0_T\in d\ell}.
\end{equation*}
This makes it difficult to sample exactly $X_t$ under $\Z$ for $t\in (0,T)$.
\end{rem} 
 \vspace{0.2cm}
 
 \subsection{Exact simulation algorithm for the solution of \eqref{eds} starting from $x$}
 Let us denote by $K$ an upper bound for $\tilde{\phi}(x)$. Following the spirit of \cite{beskos2} we can thus sample from $X_T$ using the following algorithm.
\begin{center}
\hrulefill

\vspace{-0.2cm}
\textsc{  
EXACT SIMULATION ALGORITHM FOR THE SOLUTION OF (\ref{eds}) starting from $x$.
\vspace{-0.2cm}
\begin{enumerate}
\item Simulate a Poisson Point Process with unit density on $[0,T]\times [0,K]$. The result is a random number $N$ of points of coordinates $(t_1,z_1),\dots, (t_N,z_N)$.
\item Simulate a skeleton $(\omega_{t_1},\ldots,\omega_{t_N},\omega_T)$ where $\omega\sim\hZ$.
\item If $\forall i\in \{1,\dots, N\}$ $\tilde{\phi}(\omega_{t_i})\leq z_i$ accept the skeleton. Else return to step 1.
\end{enumerate}  
}
\vspace{-0.5cm}
\hrulefill
\end{center}
\vspace{0.2cm}

This algorithm produces an exact sampling of $X_T$ under $\P$: it is the final instance $\omega_T$ of an accepted skeleton.

\vspace{0.2cm}

The main issue in the above algorithm is to sample a skeleton of the canonical process under $\hZ$ (Step 2). 

\begin{rem} (Other exact simulation algorithms)
\label{remarque-importante-2}

Other probability changes might be performed in order to try to tackle the exact simulation problem presented in the introduction. For example (though we will not prove it here) it is possible to swap to a probability measure ${\mathbb S}$ under which $X$ has the law of  some Brownian motion with a symmetric two valued drift (solution of equation \eqref{Bang-Bang-intro} in the case where $\theta_0=-\theta_1$) with some prescribed terminal law. Even though the density probability distribution of such bridges may be explicitly computed, it seems difficult to find tractable general rejection bounds for these laws.
\end{rem} 

%%It is dealt with in the next section, using the SBM with drift recalled in Section \ref{sec-rappels}. The algorithm providing the skeleton of interest will be presented in Subsection 
%%\ref{ss-samplingZ}.

\section{Recalls on the skew Brownian motion with drift}
In this section, we recall some basic facts concerning the skew Brownian motion with constant drift. Although these facts seem at first quite far away from our purpose, they will be used in the sequel in order to justify that the limit rejection algorithm presented in Section 6 returns an exact sampling under $\hZ$. At the end of this section, we give an algorithm for the simulation of bridges of SBM with constant drift that will be used as a basic building block in the sequel.
 
\subsection{The transition function of the skew Brownian motion with drift}
Let us recall that the Skew Brownian Motion (SBM) with constant drift component $\mu\in\R$, denoted by $B^{\beta,\mu}$, solves
$$
dB^{\beta,\mu}_t=dW_t+\mu dt+\beta dL^0_t(B^{\beta,\mu}).$$

This SDE with local time has a unique strong solution as soon as $|\beta|<1$ (see \cite{legall}). The process $B^{\beta,\mu}$ enjoys the homogeneous Markov property. We shall denote by $p^{\beta,\mu}(t,x,y)$ its transition function.

Let us introduce the function $v^{\beta,\mu}(t,x,y)$ defined by
\begin{equation}
\label{def-vbetamu}
\begin{array}{lll}
v^{\beta,\mu}(t,x,y)&=&
(1-\exp({-\frac{2xy}{t}}))\1_{xy>0}\\
\\
&&+(1+\sign(y)\beta)\exp({-\frac{2xy}{t}}\1_{xy>0})\big[1-\beta\mu\sqrt{2\pi t}\exp\{\frac{(|x|+|y|+t\beta\mu)^2}{2t}\}N^c(\frac{\beta\mu t+|x|+|y|}{\sqrt{t}})\big],
\end{array}
\end{equation}
where $N^c(y)=\frac{1}{\sqrt{2\pi}}\int_y^\infty e^{-z^2/2}dz$.

With this notation we can rewrite the expression of $p^{\beta,\mu}(t,x,y)$ given in \cite{etoremartinez1}.

\begin{prop}
\label{prop-new-density}
We have for all $t>0$, all $x,y\in \R$,
\begin{equation}
\label{new-pbetmu}
p^{\beta,\mu}(t,x,y)=p^{0,\mu}(t,x,y)v^{\beta,\mu}(t,x,y).
\end{equation}
\end{prop}
\begin{proof}
See \cite{etoremartinez1} (Proposition 4.7).
\end{proof}
\subsection{Bounds for the transition function of the SBM with drift}
In this paragraph, we give bounds on the transition function of the SBM with drift. These bounds will be used in the sequel to find tractable rejection bounds for our algorithm.

Let us set $\overline{\alpha}=\max(\frac{1+\beta}{2},\frac{1-\beta}{2})$
and
\begin{equation}
\label{def-gamma}
\gamma^{\beta,\mu}(t,z)=1-\beta\mu\sqrt{2\pi t}\exp(\frac{(z+t\beta\mu)^2}{2t})N^c(\frac{\beta\mu t + z}{\sqrt{t}}).
\end{equation}
We also set
\begin{equation}
\label{def-c}
c^{\beta,\mu}_{t,x}=\left\{
\begin{array}{lll}
2\overline{\alpha}&\text{if} &\beta\mu\geq 0\\
2\overline{\alpha}\gamma^{\beta,\mu}(t,|x|)&\text{if} &\beta\mu< 0.\\
\end{array}
\right.
\end{equation}

We have the following result.

\begin{lem}
\label{maj-v}
Let $(\beta,\mu)\in (-1,1)\times\R$. We have
\begin{equation}
\label{maj-dens1}
v^{\beta,\mu}(t,x,y)\leq c^{\beta,\mu}_{t,x},\quad\forall x,y\in\R.
\end{equation}
\end{lem}

\begin{proof}
 Equation \eqref{maj-dens1} comes from \eqref{new-pbetmu} and the fact that, if $\beta\mu\geq 0$, we have 
 $p^{\beta,\mu}(t,x,y)\leq 2\bar{\alpha} p^{0,\mu}(t,x,y)$ for all $x,y\in\R$, 
and if $\beta\mu<0$, we have $p^{\beta,\mu}(t,x,y)\leq 2\bar{\alpha}\gamma^{\beta,\mu}(t,|x|) p^{0,\mu}(t,x,y)$ for all $x,y\in\R$ (see, in \cite{etoremartinez1}, Lemma 5.3 and its proof).
\end{proof}

We also have the following lemma.
\begin{lem}
\label{lem-vbetamu}
Let $(\beta,\mu)\in (-1,1)\times\R$. We have
\begin{equation}
\label{maj-dens2}
v^{\beta,\mu}(t,x,y)\leq c^{\beta,\mu}_{t,y},\quad\forall x,y\in\R.
\end{equation}
\end{lem}

\begin{proof}
This comes again from \eqref{new-pbetmu}, together with the fact that $p^{\beta,\mu}(t,x,y)\leq c^{\beta,\mu}_{t,y}p^{0,\mu}(t,x,y)$,
for all $x,y\in\R$ (see again, in \cite{etoremartinez1}, Lemma 5.3, especially the proof of Equation $(5.7)$).
\end{proof}

\begin{rem}
Note that $v^{\beta,\mu}(t,x,y)>0$ and $\gamma^{\beta,\mu}(t,z)>0$ for any $t\in\R^{*,+},\;x,y,z\in\R$, even for large values of $\mu$ (see Remark 4.8 in \cite{etoremartinez1}).
\end{rem}

\subsection{Sampling bridges of the SBM with drift}
We denote by $q^{\beta,\mu}(t,T,a,b,y)$ the density defined (for  $t<T$) by 
\begin{equation*}
\P[\,B^{\beta,\mu}_t\in dy\;|\;B^{\beta,\mu}_0=a,\,B^{\beta,\mu}_T=b]=q^{\beta,\mu}(t,T,a,b,y)dy.
\end{equation*}
The function $(t,y)\mapsto q^{\beta,\mu}(t,T,a,b,y)$ is the transition density function of a bridge of a SBM with drift relating points $a$ and $b$ in $T$ unit time.

As $q^{0,\mu}(t,T,a,b,y)=q^{0,0}(t,T,a,b,y)$, by Proposition \ref{prop-new-density} we get,
\begin{equation}
\label{dens-ponts}
q^{\beta,\mu}(t,T,a,b,y)=q^{0,0}(t,T,a,b,y)\frac{v^{\beta,\mu}(t,a,y)v^{\beta,\mu}(T-t,y,b)}{v^{\beta,\mu}(T,a,b)}.
\end{equation}

Let us set 
\begin{equation}
\label{def-C}
C^{\beta,\mu}_{t,T,a,b}=\left\{
\begin{array}{lll}
4\overline{\alpha}^2&\text{if} &\beta\mu\geq 0\\
4\overline{\alpha}^2\gamma^{\beta,\mu}(t,|a|)\gamma^{\beta,\mu}(T-t,|b|)&\text{if} &\beta\mu< 0.\\
\end{array}
\right.
\end{equation}
We have
$$
\frac{q^{\beta,\mu}(t,T,a,b,y)}{q^{0,0}(t,T,a,b,y)}=\frac{C^{\beta,\mu}_{t,T,a,b}}{v^{\beta,\mu}(T,a,b)}f^{{\mathfrak B},\beta,\mu}_{a,b}(y),$$
with
\begin{equation}
\label{def-fbetamu}
f^{{\mathfrak B}, \beta,\mu}_{a,b}(y)\defi\frac{v^{\beta,\mu}(t,a,y)v^{\beta,\mu}(T-t,y,b)}{C^{\beta,\mu}_{t,T,a,b}},
\end{equation}
where the superscript ${\mathfrak B}$ appears for the word "Bridge".

Considering \eqref{def-c}, {\eqref{maj-dens1}, \eqref{maj-dens2} and \eqref{def-C} it is clear that
$$
f^{{\mathfrak B},\beta,\mu}_{a,b}(y)\leq 1,\quad\forall y\in\R.$$

We thus propose the following rejection algorithm in order to sample along $q^{\beta,\mu}(t,T,a,b,y)dy$.

\vspace{0.2cm}
\begin{center}
\hrulefill

\textsc{
Auxiliary Algorithm 1: Sampling along $q^{\beta,\mu}(t,T,a,b,y)dy$
}

\hrulefill
\textsc{
\begin{enumerate}
\item Sample a Brownian bridge $Y$ along $q^{0,0}(t,T,a,b,y)$.
\item Evaluate $$f^{{\mathfrak B},\beta,\mu}_{a,b}(Y)\leq 1.$$
\item Draw $U\sim\mathcal{U}([0,1])$. If $U\leq f^{{\mathfrak B},\beta,\mu}_{a,b}(Y)$ accept the proposed value $Y$. Else return to Step~1.
\end{enumerate}
}
\hrulefill
\end{center}
\vspace{0.2cm}

\begin{rem}
\label{rem-mu-large}
Note that the quantities $v^{\beta,\mu}$, $\gamma^{\beta,\mu}$, $c_{t,x}^{\beta,\mu}$, $C^{\beta,\mu}_{t,T,a,b}$, and $f^{{\mathfrak B},\beta,\mu}_{a,b}$ defined respectively in \eqref{def-vbetamu},\eqref{def-gamma}, \eqref{def-c}  \eqref{def-C}, and \eqref{def-fbetamu} involved in the above algorithm depend only on $\mu$ through the product $\beta\mu$. 
This computational fact gives the key ensuring the construction of the limit algorithm by convergence performed at the beginning of Section \ref{sec-conv-2}.
\end{rem}

\section{Convergence of a sequence of probability measures towards $\hZ$}
\label{sec-conv}

In this section, for any $n\in\Nat$ we denote by $X^n$ the solution of
\begin{equation}
\label{edsn}
dX^n_t=dW_t+\bar{b}(X^n_t)dt+\frac 1 n dL^0_t(X^n),\quad X^n_0=x.
\end{equation}
For the existence and uniqueness of solutions to \eqref{edsn} see \cite{legall}.

The starting point of our ideas is that, not surprisingly, we have the following strong convergence result, due to the consistency properties of SDEs with local time (see \cite{legall}).

\begin{thm}[Le Gall \cite{legall}, 1984] 
\label{thm-conv}
Let $X$ be the solution of \eqref{eds}
%modifier le label
and $(X^n)$ the sequence of solutions of \eqref{edsn}. We have for all $0<t<T$,
$$
\Esp\big[ \sup_{0\leq s\leq t}|X_s-X^n_s|\, \big]\xrightarrow[n\to\infty]{} 0.$$
\end{thm}

\begin{proof}
See the Appendix.
\end{proof}

In particular $(X^n)$ converges in law to $X$ under $\P$. This fact will allow us to construct a suitable sequence $\pare{\hZ_n}$ of probability measures converging towards $\hZ$.  

Recall the definition \eqref{eq-theta} of $\theta$. Let us set 
\begin{equation}
\label{eq-mun}
\mu_n=\frac{1+1/n}{2/n}\,\bar{b}(0+)-\frac{1-1/n}{2/n}\,\bar{b}(0-)=\frac{\bar{b}(0+)+\bar{b}(0-)}{2}+\theta n,
\end{equation}
and $b_n(x)=\bar{b}(x)-\mu_n$. 

From \eqref{edsn} we have,
$$
X^n_T=x+W^{SD,n}_T+\mu_nT+\frac 1 nL^0_T(X^n),$$
where the process $W^{SD,n}$ given by
$$
dW^{SD,n}_t=dW_t+b_n(X^n_t)dt$$
is a Brownian motion under the probability measure $\Wsdn$ defined by
\begin{equation}
\label{def-Wsdn}
\frac{d\Wsdn}{d\P}=\exp\Big\{-\int_0^Tb_n(X^n_t)dW_t - \frac 1 2 \int_0^Tb_n^2(X^n_t)dt  \Big\}.
\end{equation}
Note that the assumptions in § \ref{sub-sec:assumptions} ensure that $\Wsdn$ is well defined for all fixed $n\in {\mathbb N}$ and that the law of $X^n$ under $\Wsdn$ is the one of a SBM with drift.

Let us now set $B_n(x)=\int_0^xb_n(z)dz$.
As shown in \cite{etoremartinez1} pp 47-48, using a symmetric It\^o-Tanaka formula we can prove that for any bounded measurable functional $F:(C, \cC)\to\R$,
\begin{equation}
\label{eq:justification}
\Esp_{\P}^x[F(X^n)]=\Esp_{\Wsdn}^x\big[F(X^n)\exp\big\{B_n(X^n_T)-B_n(x)-\int_0^T\phi_n(X^n_t)dt \big\}  \big],
\end{equation}
 where $\phi_n(x)=\dfrac{b_n^2(x)+b_n'(x)+2\mu_n b_n(x)}{2}$. 
 
\begin{rem}
 \label{justification}
 Note that, because of the definition of $b_n$, there is no local time appearing in equality \eqref{eq:justification} after the application of the symmetric It\^o-Tanaka formula. This ensures that there
 is no local time involved in the exponential martingale of the probability change, which makes it tractable (from the point of view of our numerical perspective). Retrospectively, this explains why we defined $b_n$ depending on $n$ as $\bar{b}-\mu_n$ (and not just kept the initial function $\bar{b}$ to perform our computations).
 \end{rem} 
 
We see that
 $$
 \phi_n(x)=\frac{\bar{b}^2(x)+\bar{b}'(x)}{2}-\frac{\mu_n^2}{2}=\phi(x)-\frac{\mu_n^2}{2},$$
 and
 $$
 \phi_n(x)-\inf_{x\in\R}\phi_n(x)=\phi(x)-\frac{\mu_n^2}{2}-\inf_{x\in\R}(\phi(x)-\frac{\mu_n^2}{2})=\tilde{\phi}(x),$$
so that $\phi_n -\inf \phi_n$ does not depend on $n$ !
 
 Consequently, we have that 
 \begin{equation*}
 \Esp_{\P}^x[F(X^n)]\propto\Esp_{\Wsdn}^x\big[F(X^n)\exp\big\{B_n(X^n_T)-B_n(x)-\int_0^T\tilde{\phi}(X^n_t)dt \big\}  \big].
 \end{equation*}

 Let us now define the probability measure $\Z_n$ on $(C,\cC)$ by
\begin{equation}
\label{def-Zn}
\frac{d\Z_n}{d\Wsdn}(\omega)\propto\exp\big\{ B_n(X^n_T(\omega))-B_n(x) \big\},
\end{equation}
and $\hZ_n$ the probability measure induced on $(C,\cC)$ by the law of $X^n$ under $\Z_n$. Under the assumptions in~ §~\ref{sub-sec:assumptions} the probability measures $\Z_n$ and $\hZ_n$ are well defined for all fixed $n\in {\mathbb N}$. We have
\begin{equation}
\label{eq-dens-loiXn-hZn}
\Esp_{\P}^x[F(X^n)]=c_n\Esp_{\hZ_n}^x\big[F(\omega)\exp\big\{-\int_0^T\tilde{\phi}(\omega_t)dt\big\}\big],
\end{equation}
with $c_n$ a finite normalizing constant.
\vspace{0.1cm}

The law $\hZ_n$ can be well described: it is the law of a SBM with drift whose terminal position is distributed along a density $h_n$ depending on the function $B_n $ (see Subsection \ref{ss-samplingZn}). 

In \cite{etoremartinez1} we managed to sample exactly along \eqref{edsn} using skeletons under $\hZ_n$ as proposals and the function $\exp\big\{-\int_0^T\tilde{\phi}(\omega_t)dt \big\}$ for a rejection rule.

Remember that $(X^n)$ converges in law to $X$ under $\P$. Hence, comparing \eqref{eq-dens-loiX-hZ} and \eqref{eq-dens-loiXn-hZn} indicates that
the sequence of probability measures $(\hZ_n)$ converges weakly to $\hZ$. This is shown rigorously in the following subsection. Combining then the results of Subsection \ref{ss-samplingZn} and Proposition \ref{prop-delamort} will enable us to sample along 
$\hZ$ (see Section~\ref{sec-conv-2}).

\subsection{The probability measure $\hZ$ as a limit of the sequence $(\hZ_n)$}

\begin{prop}
\label{prop-convZZn}
We have
$$
\hZ_n\xrightarrow[n\to\infty]{w}\hZ.$$

\end{prop}

\begin{proof}
Let $\widehat{\P}$ (resp. $\widehat{\P}_n$) denote the probability measure induced on $(C,\cC)$ by the law of $X$ (resp. of $X^n$) under $\P$. 
It is clear from Theorem \ref{thm-conv} that $\widehat{\P}_n\xrightarrow[n\to\infty]{w}\widehat{\P}$.

Let us define $\Phi:(C, {\cal C})\to\R$ by 
$$\Phi(\omega)\defi\exp\big\{-\int_0^T\tilde{\phi}(\omega_t)dt\big\},\quad\forall\omega\in C.$$
Note that $0<\Phi(\omega)\leq 1$. Thanks to \eqref{eq-dens-loiX-hZ} and \eqref{eq-dens-loiXn-hZn} we have
$$
\frac{d\widehat{\P}}{d\hZ}(\omega)=c\,\Phi(\omega)\quad\text{and}\quad\frac{d\widehat{\P}_n}{d\hZ_n}(\omega)=c_n\,\Phi(\omega)
$$
and thus,
\begin{equation}
\label{eq-dens-hZ-hP}
\frac{d\hZ}{d\widehat{\P}}(\omega)=\frac 1 c \,\frac 1 {\Phi(\omega)}\quad\text{and}\quad\frac{d\hZ_n}{d\widehat{\P}_n}(\omega)=
\frac 1 {c_n}\,\frac 1 {\Phi(\omega)}.
\end{equation}
Under the assumptions of § \ref{sub-sec:assumptions}, the functional $\omega \mapsto 1/{\Phi(\omega)}$ is easily seen to be bounded and continuous from $(C, {\cal C})$ to $\R$ for the topology of the supreme norm. Using that $\widehat{\P}_n\xrightarrow[n\to\infty]{w}\widehat{\P}$ we see that
\begin{equation}
\label{eq-dens-hZ-hP-2}
\int_C\frac{\widehat{\P}_n(d\omega)}{\Phi(\omega)}\xrightarrow[n\to\infty]{}\int_C\frac{\widehat{\P}(d\omega)}{\Phi(\omega)}.
\end{equation}
Since $\hZ_n$ is a probability measure on $(C, {\cal C})$,  we also have $1=\hZ_n(C) = \hZ(C)$. In view of \eqref{eq-dens-hZ-hP} and \eqref{eq-dens-hZ-hP-2} this implies that necessarily $(1/c_n)_n$ is a convergent sequence and that 
\begin{equation}
\label{eq-cn-c}
\lim_n \frac 1 {c_n}=\frac 1 c.
\end{equation}

Therefore, for any bounded and continuous funcional $\omega \mapsto F(\omega)$ from $(C, {\cal C})$ to $\R$, 
$$
\int_C F(\omega)d\hZ_n = \frac{1}{c_n}\int_C F(\omega)\frac{\widehat{\P}_n(d\omega)}{\Phi(\omega)}\xrightarrow[n\to\infty]{}\frac{1}{c}\int_C F(\omega)\frac{\widehat{\P}(d\omega)}{\Phi(\omega)}= \int_C F(\omega)d\hZ$$
and the result follows.
\end{proof}

\subsection{Sampling a skeleton under $\hZ_n$}
\label{ss-samplingZn}

We have the following proposition.

\begin{prop}
For any $n\in\Nat$ the law $\hZ_n$ is the one of a SBM $B^{\frac 1 n,\mu_n}$ with drift  $\mu_n$ conditionally on $B^{\frac 1 n,\mu_n}_T\sim h_n$
with
$$
h_n(y)=C_n\exp\pare{B_n(y)-B_n(x)} p^{\frac 1 n,\mu_n}(T,x,y),
$$
where $C_n$ is the normalizing constant such that $\int h_n(y)dy = 1$. 
\end{prop}

\begin{proof}
See \cite{beskos2}.
\end{proof}

Let $n_0$ be fixed and $0=t_0<t_1<\ldots<t_{n_0}<T$. Set $y_0=x$ to simplify the notations, the law of $(\omega_{t_1},\ldots,\omega_{t_{n_0}},\omega_T)$ under $\hZ_n$ is given by
\small
\begin{equation}
\label{eq-cond}
h_n(y)\prod_{i=0}^{n_0-1} q^{\frac 1 n,\mu_n}(t_{i+1}-t_{i},T-t_{i},y_{i},y,y_{i+1})dy_1\ldots dy_{n_0}dy.
%&q^{\frac 1 n,\mu_n}(t_1,T,x,y,y_1)\times q^{\frac 1 n,\mu_n}(t_2-t_1,T-t_1,y_1,y,y_2)\times\ldots\times
%q^{\frac 1 n,\mu_n}(t_{n_0}-t_{n_0-1},T-t_{n_0-1},y_{n_0-1},y,y_{n_0})h_n(y)dy_1\ldots dy_{n_0}dy.
\end{equation}
\normalsize

Once $\omega_T$ has been sampled along $h_n(y)dy$, we can sample $\omega_{t_1}$ along $q^{\frac 1 n,\mu_n}(t_1,T,x,\omega_T,y_1)dy_1$ and each $\omega_{t_{i+1}}$ along 
$q^{\frac 1 n,\mu_n}(t_{i+1}-t_i,T-t_i,\omega_{t_i},\omega_T,y_{i+1})dy_{i+1}$, using the Auxiliary Algorithm 1. 

\vspace{0.4cm}
In order to sample along $h_n(y)dy$ we make use of the following considerations.
We have
\begin{align*}
h_n(y)=&C_n\exp\big( B_n(y)-B_n(x) \big)p^{0,\mu_n}(T,x,y)v^{\frac{1}{n},\mu_n}(T,x,y)\\
=&C_n\exp\big( -\mu_n(y-x)+\int_x^y\bar{b}(z)dz\big)\times\exp\big(+\mu_n(y-x)-\frac{\mu_n^2}{2}T\big)p^{0,0}(T,x,y)v^{\frac{1}{n},\mu_n}(T,x,y)\\
=&C_ne^{-\frac{\mu_n^2}{2}T}\exp\big( B(y)-B(x) \big)v^{\frac{1}{n},\mu_n}(T,x,y)p^{0,0}(T,x,y).
\end{align*}

Recall that $M$ denotes an upper bound for the function $z\mapsto |\bar{b}|(z)$ (see \eqref{eq:M} of our assumptions in § \ref{sub-sec:assumptions}). Then, using the result of Lemma \ref{maj-v} and performing easy computations, we easily see that for any $0<\delta <1$~:~
$$
\frac{h_n(y)}{p^{0,0}(T/(1-\delta),x,y)}= C_n\frac{e^{T^2M^2/\delta}}{\sqrt{1-\delta}}e^{-\frac{\mu_n^2}{2}T}c^{\frac 1 n,\mu_n}_{T,x}\,f_{\delta}^{{\mathfrak h}, \frac{1}{n}, \mu_n}(y),$$
with
$$
f_{\delta}^{{\mathfrak h}, \frac{1}{n}, \mu_n}(y)=\sqrt{1-\delta}\,\exp\pare{B(y)-B(x) -\frac{T^2M^2}{\delta}}\frac{p^{0,0}(T,x,y)}{p^{0,0}(T/(1-\delta),x,y)}\frac{v^{\frac 1 n,\mu_n}(T,x,y)}{c^{\frac 1 n,\mu_n}_{T,x}}.$$

 Using \eqref{maj-dens1}  one may easily check that $f_{\delta}^{{\mathfrak h}, \frac{1}{n}, \mu_n}(y)\leq 1$ for any $y\in\R$. One might then optimize w.r.t. $\delta\in (0,1)$ in order to find $f_{\delta}^{{\mathfrak h}, \frac{1}{n}, \mu_n}$ closest to $1$. 
 
 Let us set for simplicity, $f^{{\mathfrak h}, \frac{1}{n}, \mu_n} = f_{1/2}^{{\mathfrak h}, \frac{1}{n}, \mu_n}$. We deduce therefore the following procedure in order to sample along $h_n(y)dy$.

\vspace{0.2cm}
\begin{center}
\hrulefill

\textsc{
Auxiliary Algorithm 2: Sampling along $h_n(y)dy$
}

\hrulefill

\textsc{
\begin{enumerate}
\item Sample $Y\sim\cN(x,2T)$.
\item Evaluate $$f^{{\mathfrak h}, \frac{1}{n}, \mu_n}(Y)\leq 1.$$
\item Draw $U\sim\mathcal{U}([0,1])$. If $U\leq f^{{\mathfrak h}, \frac{1}{n}, \mu_n}(Y)$ accept the proposed value $Y$. Else return to Step~1.
\end{enumerate}
}

\hrulefill
\end{center}

\section{Direct exact sampling of a skeleton under $\hZ$ (Step 2 of the Exact Simulation Algorithm)}
\label{sec-conv-2}

 Proposition \ref{prop-delamort} will now play a crucial role. 
 
 Recall the definition \eqref{eq-theta} of $\theta$. Let us denote
$$
v^{\theta}(t,x,y)=
(1-e^{-2xy/t})\1_{xy>0}
+e^{-2xy/t}\big[1-\theta\sqrt{2\pi t}\exp\{\frac{(|x|+|y|+t\theta)^2}{2t}\}N^c(\frac{\theta t+|x|+|y|}{\sqrt{t}})\big],
$$
$$
\gamma^{\theta}(t,z)=1-\theta\sqrt{2\pi t}\exp(\frac{(z+t\theta)^2}{2t})N^c(\frac{\theta t + z}{\sqrt{t}}),
$$
$$
c^{\theta}_{t,x}=\left\{
\begin{array}{lll}
1&\text{if} &\theta\geq 0\\
\gamma^{\theta}(t,|x|)&\text{if} &\theta< 0,\\
\end{array}
\right.
\quad
\text{and}
\quad
C^{\theta}_{t,T,a,b}=\left\{
\begin{array}{lll}
1&\text{if} &\theta\geq 0\\
\gamma^{\theta}(t,|a|)\gamma^{\theta}(T-t,|b|)&\text{if} &\theta< 0.\\
\end{array}
\right.
$$
Remember our definitions \eqref{def-vbetamu},\eqref{def-gamma},\eqref{def-c} and  \eqref{def-C} and Remark \ref{rem-mu-large}. It is clear from \eqref{eq-mun} that  $\frac{1}{n}\mu_n\rightarrow \theta$ (as $n$ tends to $+\infty$), so that we  have,
$$
\begin{array}{ll}
v^{\frac 1 n,\mu_n}(t,x,y)\xrightarrow[n\to\infty]{}v^\theta(t,x,y)&\forall (t,x,y)\in \R^+\times \R\times \R,\\
\gamma^{\frac 1 n,\mu_n}(t,z)\xrightarrow[n\to\infty]{}\gamma^\theta(t,z)&\forall (t,z)\in \R^+\times \R,\\
c^{\frac 1 n,\mu_n}_{t,x}\xrightarrow[n\to\infty]{}c^{\theta}_{t,x}&\forall (t,x)\in \R^+\times \R,\\
C^{\frac 1 n,\mu_n}_{t,T,a,b}\xrightarrow[n\to\infty]{}C^{\theta}_{t,T,a,b}&\forall (t,T,a,b)\in \R^+\times \R^+\times \R\times \R.\\
\end{array}
$$

Let us now examine the sequence $(f_\delta^{{\mathfrak h}, \frac{1}{n}, \mu_n})$ of the rejection functions used in the Auxiliary Algorithm 2. From the same reasons as above, it is clear that $(f_\delta^{{\mathfrak h}, \frac{1}{n}, \mu_n})$
converges towards
$$
f_{\delta}^{{\mathfrak h}, \theta}(y)=\sqrt{1-\delta}\,\exp\pare{B(y)-B(x) -\frac{T^2M^2}{\delta}}\frac{p^{0,0}(T,x,y)}{p^{0,0}(T/(1-\delta),x,y)}\frac{v^{\theta}(T,x,y)}{c^{\theta}_{T,x}}\leq 1,$$
this convergence being dominated. 
Thus, applying the result of Proposition \ref{prop-delamort}, the sequence of laws $(h_n(y)dy)$ converges to some limit law $h_\theta(y)dy$.

In the same manner, for any fixed $a,b\in\R$, the sequence $(f^{{\mathfrak B}, \frac 1 n,\mu_n}_{a,b})$ of rejection functions used in Auxiliary Algorithm~1 converges towards
 $$f^{{\mathfrak B},\theta}_{a,b}(y)\defi\frac{v^{\theta}(t,a,y)v^{\theta}(T-t,y,b)}{C^{\theta}_{t,T,a,b}}\leq 1,$$
 this convergence being dominated.
 
 Consequently, the law $q^{\frac 1 n,\mu_n}(t,T,a,b,y)dy$ converges towards a limit law $q^{\theta}(t,T,a,b,y)dy$.
 
 Let again $n_0$ be fixed and $0<t_1<\ldots<t_{n_0}<T$. Passing to the limit in \eqref{eq-cond} we get that the law of 
 $(\omega_{t_1},\ldots,\omega_{t_{n_0}},\omega_T)$ under $\hZ_n$ converges (with $y_0=x$) towards
 \small
 \begin{equation}
 \label{eq-condlim}
 h_\theta(y)\prod_{i=0}^{n_0-1} q^{\theta}(t_{i+1}-t_{i},T-t_{i},y_{i},y,y_{i+1})dy_1\ldots dy_{n_0}dy.
 %%q^{\theta}(t_1,T,x,y,y_1)\times q^{\theta}(t_2-t_1,T-t_1,y_1,y,y_2)\times\ldots\times
%%q^{\theta}(t_{n_0}-t_{n_0-1},T-t_{n_0-1},y_{n_0-1},y,y_{n_0})h_\theta(y)dy_1\ldots dy_{n_0}dy.
\end{equation}

\normalsize
Consequently, from Proposition \ref{prop-convZZn}, we conclude that the law given by \eqref{eq-condlim} is nothing else than the law of $(\omega_{t_1},\ldots,\omega_{t_{n_0}},\omega_T)$ under $\hZ$.

 Using again Proposition \ref{prop-delamort} and the above considerations we can propose the expected algorithm in order to sample skeletons under $\hZ$. It will use the two following Limit Auxiliary Algorithms.

 \vspace{0.2cm}
\begin{center}
\hrulefill

 \textsc{
Limit Auxiliary Algorithm 1: Sampling along $h_\theta(y)dy$
}

\hrulefill
\textsc{
\begin{enumerate}
\item Sample $Y\sim\cN(x,2T)$.
\item Evaluate $$f^{{\mathfrak h},\theta}(Y)\leq 1.$$
\item Draw $U\sim\mathcal{U}([0,1])$. If $U\leq f^{{\mathfrak h}, \theta}(Y)$ accept the proposed value $Y$. Else return to Step 1.
\end{enumerate}
}

\hrulefill
\end{center}
 \vspace{0.2cm}
\begin{center}
\hrulefill

\textsc{
Limit Auxiliary Algorithm 2: Sampling along $q^{\theta}(t,T,a,b,y)dy$
}

\hrulefill
\textsc{
\begin{enumerate}
\item Sample a Brownian bridge $Y$ along $q^{0,0}(t,T,a,b,y)$.
\item Evaluate $$f^{\mathfrak{B},\theta}_{a,b}(Y)\leq 1.$$
\item Draw $U\sim\mathcal{U}([0,1])$. If $U\leq f^{\mathfrak{B},\theta}_{a,b}(Y)$ accept the proposed value $Y$. Else return to Step 1.
\end{enumerate}
}

\hrulefill
\end{center}
\vspace{0.2cm}

 \vspace{0.2cm}
\begin{center}
\hrulefill

 \textsc{
Performing Step 2 of the Exact Simulation Algorithm.\\
 Sampling $(\omega_{t_1},\ldots,\omega_{t_{n_0}},\omega_T)$ under $\hZ$ (starting from $x$)
}

\hrulefill

\textsc{
\begin{enumerate}
\item Sample $\omega_T$ along $h_\theta(y)dy$ using the Limit Auxiliary Algorithm 1.
\item Sample $\omega_{t_1}$ along $q^\theta(t_1,T,x,\omega_T,y)dy$ using the Limit Auxiliary Algorithm 2.
\item For $i=2,\ldots,n_0$, sample $\omega_{t_{i+1}}$ along $q^\theta(t_{i+1}-t_i,T-t_i,\omega_{t_i},\omega_T,y)dy$ using the Limit Auxiliary Algorithm 2.
\end{enumerate}
}

\hrulefill
\end{center}

\section{Numerical Experiments}
\subsection{Exact simulation of a Brownian motion with two-valued (or alternate) drift}
In this paragraph, we choose to exhibit numerical results obtained with the exact limit algorithm for the simplest non-trivial cases
\begin{equation*}
dX_t=dW_t\pm\sgn(X_t)dt,\quad X_0=0,
\end{equation*}
corresponding to either $\theta_0=-\theta_1=\pm1$ in \eqref{Bang-Bang-intro} ($\bar{b}(y) = \pm\sgn(y)$ in \eqref{eds-intro}).
Indeed, in this symmetric case a benchmark is provided by the explicit and computable density of $X_T$ given in \cite{kara} p. 440-441.

We draw the renormalized histogram of $10^6$ samples of $X_T$ and compare it to the explicit density of $X_T$ (Figure \ref{fig1} for the outgoing case $\theta_0=1$ and Figure \ref{fig2} for the incoming case $\theta_0=-1$).

In the non-symmetric case we can still use our limit algorithm but the density of $X_T$ becomes less explicit (see formula (6.5.12)
in \cite{kara}). Thus we will use as a benchmark the renormalized histogram of $10^6$ samples of  $X_T^\Delta$, where $(X^\Delta)$ denotes an Euler Scheme with time step $\Delta=T.10^{-5}$. We chose $\theta_0=2$, $\theta_1=-1$, $T=1$ and $x=0.0$.
We plot the corresponding renormalized histograms on Figure \ref{fig3}.

In Table \ref{tab1} we report the CPU times needed to get the $10^6$ samples, with the exact limit algorithm and the Euler scheme.
Programs were written in C-language and executed on a personal computer equipped with an Intel Core 2 duo processor, running at $2.23$ Ghz.
We report in Table \ref{tabrej1} the acceptance ratios.

On this example the acceptance ratios are good and the exact method is nearly four times faster than the Euler scheme with time step 
$\Delta=T.10^{-5}$.

\begin{table}
 \begin{center}
\begin{tabular}{ccc}
\hline
 &Exact & Euler   \\
\hline
CPU times & 2111s& 9521s\\
\hline
\end{tabular}
\caption{CPU times for $10^6$ simulations of a Brownian motion with two-valued drift with $\theta_0=2$ and $\theta_1=-1$  ($x=0.0$ and $T=1$).}
\label{tab1} 
\end{center}
\end{table}

\begin{table}
 \begin{center}
\begin{tabular}{cccc}
\hline
 & Exact Algorithm  &  Bridges  \\
 \hline
Acceptance Ratio&  20.4\%    &   58,6\% \\
\hline
\end{tabular}
\caption{Acceptance ratios for the case of a Brownian motion with two-valued drift with $\theta_0=2$ and $\theta_1=-1$  ($x=0.0$ and $T=1$).}
\label{tabrej1} 
\end{center}
\end{table}

\begin{figure}
\begin{center}
\epsfig{file=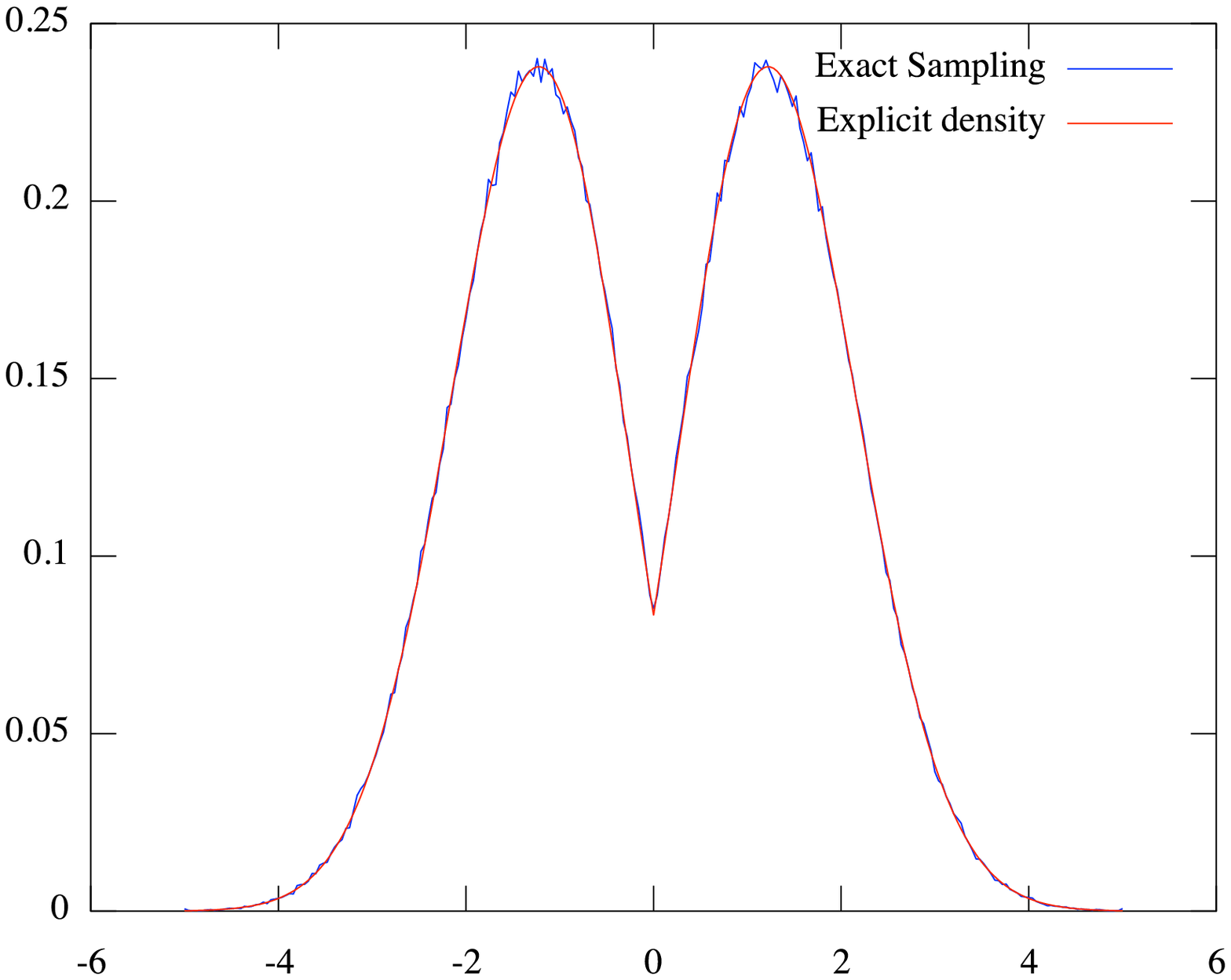,width=12cm}
\caption{  Brownian motion with two-valued drift, case $\theta_0=-\theta_1=1$ ($T=1$).  }
\label{fig1}
\end{center}
\end{figure}

\begin{figure}
\begin{center}
\epsfig{file=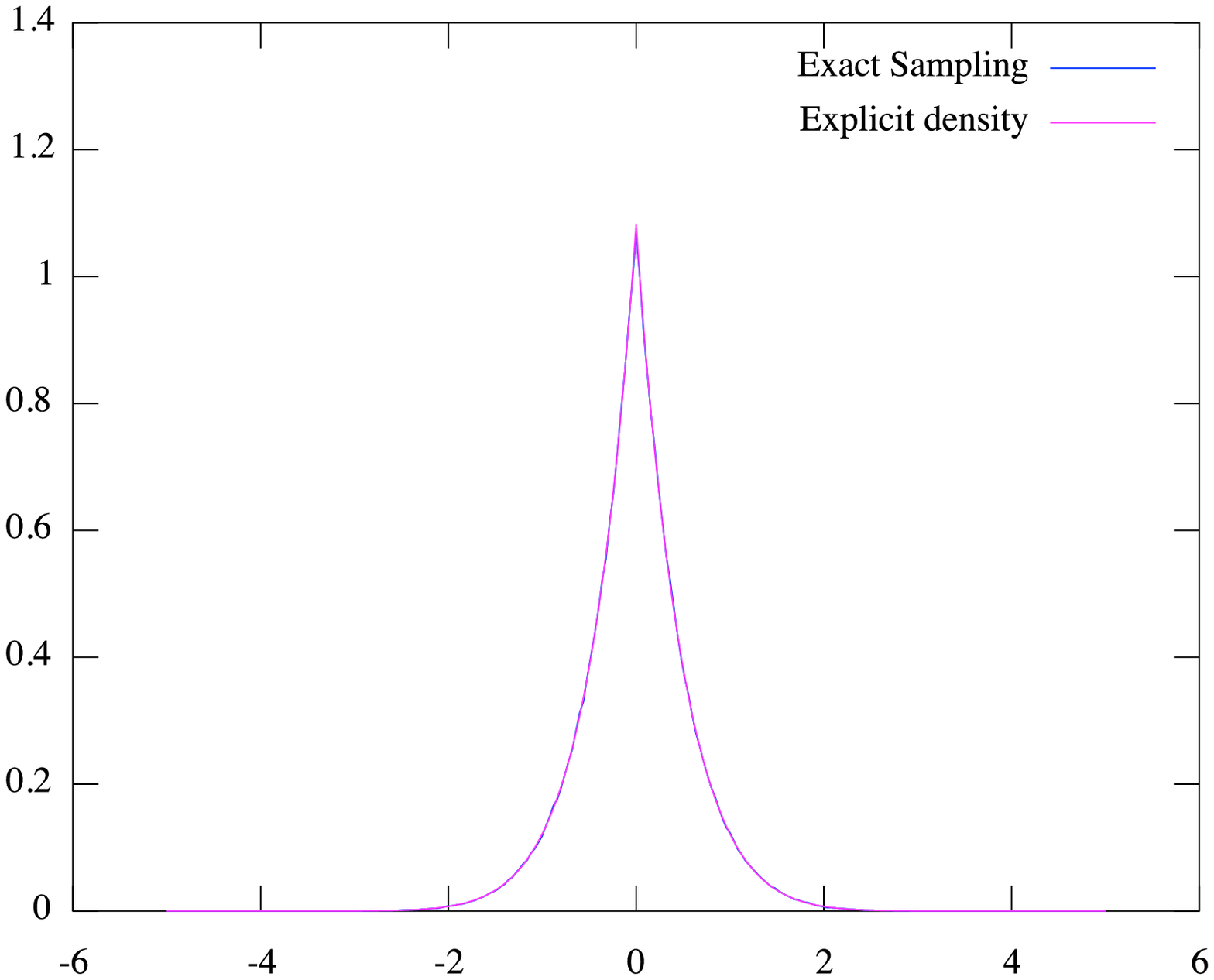,width=12cm}
\caption{  Brownian motion with two-valued drift, case $\theta_0=-\theta_1=-1$ ($T=1$).  }
\label{fig2}
\end{center}
\end{figure}

\begin{figure}
\begin{center}
\epsfig{file=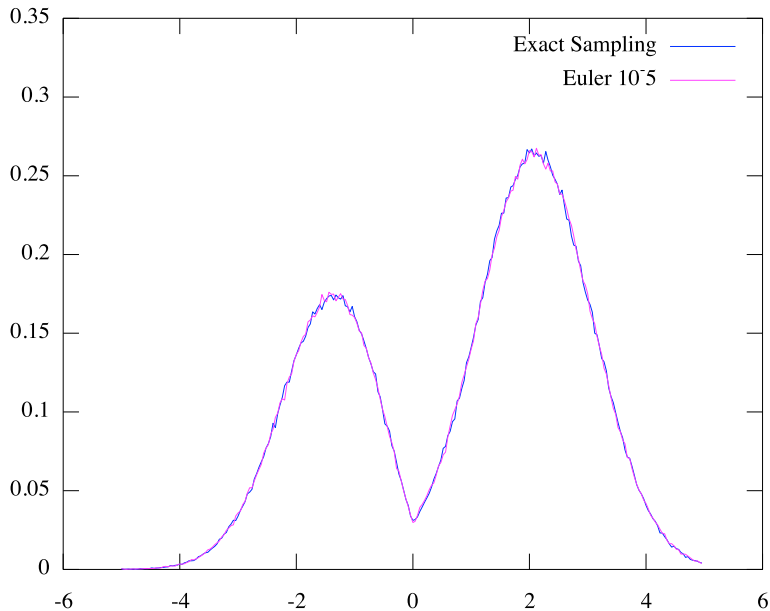,width=12cm}
\caption{ Limit algorithm v.s. Euler Scheme for  Brownian motion with two-valued drift with $\theta_0=2$ and $\theta_1=-1$  ($x=0.0$ and $T=1$).}
\label{fig3}
\end{center}
\end{figure}

\subsection{Exact simulation of an SDE with a discontinuous drift coefficient}

We consider now the SDE \eqref{eds-intro}
with
\begin{equation}
\label{example-barb}
\bar{b}(x)=\left\{
\begin{array}{lll}
-\frac{\pi}{2}\cos\big( \frac{\pi}{5}x \big)&\text{if}&x\geq 0\\
\\
\frac{3\pi}{2}-\frac{\pi}{2}\cos\big( \frac{\pi}{5}x \big)&\text{if}&x< 0.\\
\end{array}
\right.
\end{equation} 
Let $0<T<\infty$. We wish to sample along $X_T$.

We have $\theta=-3\pi/4$ and
$$
\tilde{\phi}(x)=\frac{\bar{b}^2(x)+\bar{b}'(x)}{2}+\frac{\pi^2}{20}.$$
We take $K=2\pi^2+\frac{\pi^2}{10}$ as an upper bound for $\tilde{\phi}$.
This allows to use the limit Algorithm.

Figure \ref{fig3} shows a comparison between a renormalized histogram of $10^6$ samples of $X_T$ obtained with the exact limit algorithm, and a renormalized histogram of $10^6$ samples of  $X_T^\Delta$, where $(X^\Delta)$ denotes an Euler Scheme with time step $\Delta$. We chose $x=0.0$, $T=1$ and time-steps $\Delta=T.10^{-2}$ and $\Delta=T.10^{-5}$.

\begin{figure}
\begin{center}
\epsfig{file=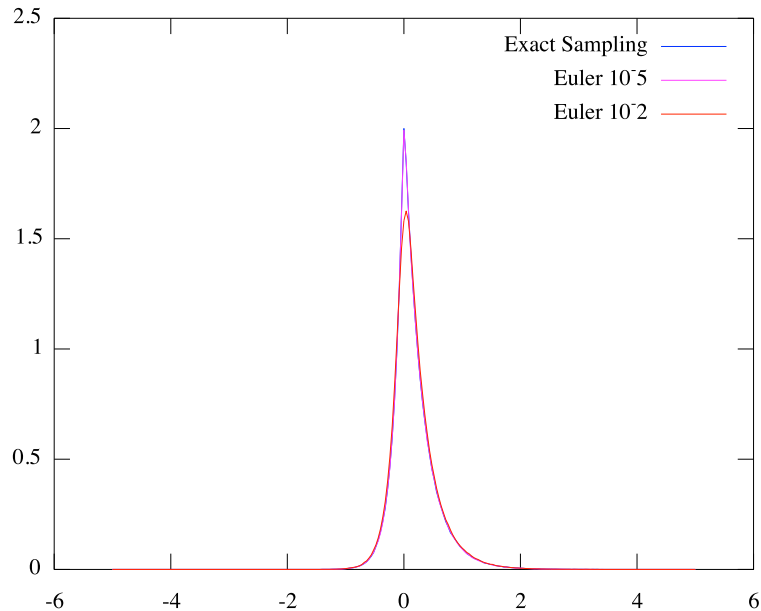,width=12cm}
\caption{ Limit algorithm v.s. Euler Scheme for the case where $\bar{b}$ is given by \eqref{example-barb} ($x=0.0$ and $T=1$).}
\label{fig4}
\end{center}
\end{figure}

In Table \ref{tab2} we report the CPU times needed to get the $10^6$ samples, with the exact limit algorithm and the Euler scheme (and, for the later one, with the different time steps we have used).
We report in Table \ref{tabrej2} the acceptance ratios.

On this example the exact simulation is competitive, compared to schemes with very fine grids.

\begin{table}
 \begin{center}
\begin{tabular}{ccc}
\hline
 &Exact & Euler   \\
&& ($\Delta t=10^{-n}$, $n=2,5$)\\
\hline
CPU times &11813s& 20s\\
&  &12952s\\
\hline
\end{tabular}
\caption{CPU times for $10^6$ simulations of $X_T$ for the case where $\bar{b}$ is given by \eqref{example-barb} ($x=0.0$ and $T=1$).}
\label{tab2} 
\end{center}
\end{table}

\begin{table}
 \begin{center}
\begin{tabular}{cccc}
\hline
 & Exact Algorithm  &  Bridges  \\
 \hline
Acceptance Ratio&  3.6\%    &   50,7\% \\
\hline
\end{tabular}
\caption{Acceptance ratios for the case where $\bar{b}$ is given by \eqref{example-barb} ($x=0.0$ and $T=1$).}
\label{tabrej2} 
\end{center}
\end{table}

\newpage

\section{Appendix}

\begin{proof}[Proof of Theorem \ref{thm-conv}]
We use the notations of \cite{legall}. Using the Occupation times formula we can rewrite Equation \eqref{edsn} as
$$
dX^n_t=dW_t+\int_\R\nu_n(dy)\,dL^y_t(X^n),\quad X^n_0=x,
$$
with $\nu_n(dy)=\bar{b}(y)dy+\frac 1 n\delta_0(dy)$. Lemma 2.1 in \cite{legall} asserts that there is for each $n\in\Nat$ a function
$f_{\nu_n}$, unique up to a multiplicative constant, satisfying $f_{\nu_n}'(dy)+(f_{\nu_n}(y)+f_{\nu_n}(y-))\nu_n(dy)$, where the notation $f_{\nu_n}'(dy)$ is for the derivative of $f_{\nu_n}$ in the generalized sense. Lemma 2.1 in \cite{legall} also asserts that if we require that 
$f_{\nu_n}(x)\to 1$ as $x\to -\infty$ then,
$$
f_{\nu_n}(y)=\exp\big( -2\int_{-\infty}^y\bar{b}(z)dz \big)\times\1_{y\geq 0}\frac{1- 1/n}{1+ 1/n}.$$
The sequence of functions $(f_{\nu_n})_n$ clearly converges point-wise to $f(y)=\exp\big( -2\int_{-\infty}^y\bar{b}(z)dz \big)$. By dominated convergence we have for all $K>0$ that $\int_{-K}^K|f_{\nu_n}-f|(y)dy\to 0$ as $n\to\infty$.
Thus Theorem 3.1 in \cite{legall} asserts that 
$$
\Esp\big[ \sup_{0\leq s\leq t}|X_s-X^n_s|\, \big]\xrightarrow[n\to\infty]{} 0,$$
with $X$ the solution of
$$
dX_t=dW_t+\int_\R\nu(dy)\,dL^y_t(X),\quad X_0=x,$$
where $\nu(dy)=-\frac{f'(dy)}{f(y)+f(y-)}=-\frac{1}{2}\frac{f'(y)}{f(y)}dy=\bar{b}(y)dy$. That is to say $X$ is the solution of \eqref{eds}.

\end{proof}

\bibliographystyle{plain}
\bibliography{TVD.bib}

\end{document}